\documentclass[smallextended]{svjour3}       
\smartqed

\usepackage{graphicx}
\usepackage{amsmath}
\usepackage{amsfonts}
\usepackage{amssymb}
\usepackage{cite}
\usepackage{hyperref}
\usepackage{enumerate}
\usepackage{float}

\renewcommand{\Bbb}{\mathbb{B}}

\newcommand{\Rbb}{\mathbb{R}}
\newcommand{\Sbb}{\mathbb{S}}

\newcommand{\paren}[1]{\left(#1\right)}
\newcommand{\parenn}[1]{\left[#1\right]}
\newcommand{\parennn}[1]{\left\{#1\right\}}
\newcommand{\parennnn}[1]{\left|#1\right|}

\renewcommand{\Im}{{\rm Im}}
\renewcommand{\Re}{{\rm Re}}

\newcommand{\Ball}{{\Bbb}}

\newcommand{\set}[2]{\left\{#1\,\left|\,#2\right.\right\}}

\newcommand{\mmap}[3]{#1:\,#2\rightrightarrows #3\,}

\newcommand{\ip}[2]{\left\langle #1,#2\right\rangle}

\newcommand{\norm}[1]{\left\|#1\right\|}

\DeclareMathOperator{\Id}{Id}

\DeclareMathOperator{\argmin}{argmin\,}

\DeclareMathOperator{\dist}{dist}

\DeclareMathOperator{\Fix}{\mathsf{Fix}\,}

\newcommand{\ncone}[1]{{N}_{#1}}

\newcommand{\fncone}[1]{{\widehat{N}}_{#1}}

\newcommand {\Limsup} {\mathop{{\rm Lim\,sup}\,}}

\hypersetup{colorlinks=true,urlcolor=blue,citecolor=blue}
\usepackage{stackrel}

\usepackage[shortlabels]{enumitem}
\usepackage{epsfig}
\usepackage[final]{pdfpages}
\usepackage{float}
\newcommand{\TO}[1]{\stackrel{#1}{\to}}

\newcommand{\fft}{\mathcal{F}}

\DeclareMathOperator{\diag}{diag}

\newcommand*{\ee}{\ensuremath{\text{e}}}

\newcommand*{\jj}{\ensuremath{\text{j}}}

\begin{document}

\title{Convex combination of alternating projection and Douglas--Rachford operators for phase retrieval
\thanks{This project has received funding from the ECSEL Joint Undertaking (JU) under grant agreement No. 826589. The JU receives support from the European Union's Horizon 2020 research and innovation programme and Netherlands, Belgium, Germany, France, Italy, Austria, Hungary, Romania, Sweden and Israel.}
}
\titlerunning{Convex combination of AP and DR operators for phase retrieval}

\author{Nguyen Hieu Thao, Oleg Soloviev and Michel Verhaegen}


\institute{Nguyen Hieu Thao\at
Delft Center for Systems and Control,
Delft University of Technology, 2628CD Delft, The Netherlands. Department of Mathematics, School of Education, Can Tho University, Can Tho, Vietnam.
\email{h.t.nguyen-3@tudelft.nl, nhthao@ctu.edu.vn}\\
Oleg Soloviev\at
Delft Center for Systems and Control,
Delft University of Technology,
2628CD Delft, The Netherlands.
Flexible Optical B.V., Polakweg 10-11, 2288 GG Rijswijk, The Netherlands.
\email{o.a.soloviev@tudelft.nl}\\
Michel Verhaegen\at
Delft Center for Systems and Control,
Delft University of Technology,
2628CD Delft, The Netherlands.
\email{m.verhaegen@tudelft.nl}
}

\date{Received: date / Accepted: date}

\maketitle

\begin{abstract}

We present the convergence analysis of convex combination of the alternating projection and Douglas--Rachford operators for solving the phase retrieval problem.
New convergence criteria for iterations generated by the algorithm are established by applying various schemes of numerical analysis and exploring both physical and mathematical characteristics of the phase retrieval problem.
Numerical results demonstrate the advantages of the algorithm over the other widely known projection methods in practically relevant  simulations.


\end{abstract}

\section{Introduction}\label{s:intro}

\emph{Phase retrieval} is an inverse problem of recovering the phase of a complex signal from its measured amplitude. 
It appears in various modifications in many scientific and engineering fields, including astronomy imaging \cite{DaiFie87,HarTho00}, X-ray crystallography \cite{Mil90,Har93},  microscopy \cite{Arr99,KimZhoGodPop16} and adaptive optics \cite{Mugnier2006,AntVer15,VisVer13,VisBruVer16}.
An important application of phase retrieval in optics is to quantify the properties of an imaging system via its generalized pupil function \cite{Jan02,BraDirJan02,DoeThaVer18,PisGupSolVer18}. 
The fundamental advantage of this approach compared to those using intensity point spread functions (PSFs) or intensity optical transfer functions is that it is modifiable and automatically takes the specific characteristics of the imaging system under investigation into account.
In adaptive optics, one needs to know the phase of the optical field in the system aperture to be able to compensate for an optical aberration, and the phase retrieval is a basis for a wide class of focal-plane based wavefront sensors.

Since the fundamental work \cite{Say52} of Sayre in 1952, which reveals that the phase of a scattered wave can be recovered from the recorded images at and between Bragg peaks of a diffracted wavefront, a wide variety of solution methods for phase retrieval has been proposed and developed. 
For an overview of phase retrieval algorithms, we refer the reader to the papers \cite{Fie13,SheEldCohChaMiaSeg15,Luk17,LukSabTeb19}. 
Direct methods usually require insights about the crystallographic structure to recover the missing phase \cite{Hau86}. 
Such a structural information is not only costly in terms of computational complexity but also sensitive to noise and approximation, for example, due to physical limitation or model deviation. 
As a consequence, this approach lacks practicability and becomes less popular in practice. 
The second class of solution algorithms relies on the fact that phase retrieval problems can be reformulated as linear equations with rank and positive semidefinite constraints in higher dimensional spaces. 
Well known examples of this algorithm class are \emph{MaxCut}~\cite{GoeWil95}, \emph{PhaseCut} \cite{WalAspMal15} and \emph{PhaseLift} \cite{CanEldStrVor13,CanStrVor13}.
This convex relaxation approach requires the \emph{matrix lifting} step which is computationally demanding and hence not suitable for large-scale problems.
The most popular class of phase retrieval methods is based on projections and pioneered by the work of Gerchberg and Saxton \cite{GerSax72}, which deals with phase retrieval given a single PSF image and the amplitude of the complex signal, which in the sequel will be referred to as the \emph{amplitude constraint} in order to clearly differentiate it from the \emph{intensity constraints} determined by data images.
The need to deal with more and more phase retrieval models, for example, incorporating various types of \emph{a priori} constraint \cite{Fie82}, being given multiple images and involving regularization schemes, has given rise to a wide range of solution methods in this class.
It was recently observed by Luke et al. \cite{LukSabTeb19} that this class of methods actually outperforms the other classes of phase retrieval algorithms.

In light of \cite{LevSta84,BauComLuk02,LukBurLyo02}, phase retrieval can be interpreted as mathematical feasibility problems and, as a consequence, all algorithmic schemes for set feasibility can be adapted for phase retrieval. 
The current research is devoted to that topic. 
The alternating projection (AP) and the Douglas--Rachford (DR) algorithms are perhaps the most widely known solution methods for set feasibility and have served as a basis for a wide range of modifications and regularizations, see, for example, \cite{BauMou17,KruLukNgu18}.
It has been observed that AP is stable, always convergent and to some extent able to suppress noise but it may get stuck at undesired local minima and the convergence speed can be very slow \cite{Fie82}.
In contrast, DR can be faster in convergence and better in escaping from bad local minima but less robust against noise and model deviation \cite{Luk08}.
As a result, this algorithm can not be naively applied to practical problems which intrinsically involve noise and model approximation.
This fact has motivated a number of its efficient relaxation schemes such as the usage of the Krasnoselski--Mann relaxation, the Fienup's hybrid input-output (HIO) algorithm \cite{Fie82}, the \emph{relaxed averaged alternating reflections} (RAAR) algorithm \cite{Luk05,Luk08} and the DRAP algorithm \cite{Tha18}.

In this paper, we analyze the DRAP algorithm for solving the phase retrieval problem for the first time after having observed that it appears to be the most efficient algorithm for the problem setting under consideration, see Section \ref{s:numerical experiment}.
Interestingly, DRAP mathematically coincides with the convex combination of the AP and DR operators in the phase retrieval setting.
As a result, DRAP admits two mathematically equivalent descriptions (see \eqref{DRAP} and \eqref{DRAP_0} in Section \ref{s:projection algorithm}).
The first one ensures that its computational complexity is only approximate to that of each of the constituent operators and thus it is used for numerical implementation.
The second description as a convex combination of the AP and the DR operators exhibits a concrete connection to the fundamental projection algorithms and hence it is intuitively better situated on the map of projection methods (see Remark \ref{r:2 expressions}).

The main contribution of this paper is the convergence analysis of the DRAP algorithm for solving the phase retrieval problem.
First, using the analysis approach initiated by Chen and Fannjiang \cite{CheFan18}, we establish a convergence criterion for DRAP (Theorem \ref{t:LA}), which extends the convergence result of the DR algorithm formulated in that paper.
It is worth mentioning here that extending a convergence criterion for DR to a corresponding one for its relaxations such as HIO, RAAR and DRAP algorithms is not trivial\footnote{For example, similar criterion for RAAR was proved in \cite{LiZho17} while the one for HIO remains unknown.}.
Proposition \ref{p:remove_chi} extends the applicable scope of this type of convergence results\footnote{Including the criteria for DR and RAAR algorithms formulated in \cite{CheFan18} and \cite{LiZho17}, respectively.} to cover also phase retrieval problems with \textit{amplitude constraint}.
Second, applying the analysis scheme developed by Luke et al. \cite{LukNguTam18}, we establish another convergence criterion for the DRAP algorithm (Theorem \ref{t:VA}) by integrating the physical properties of the phase retrieval problem \cite{Luk08} into the earlier known results for DRAP \cite{Tha18}.
Recall that the analysis of the latter article involves only abstract mathematical notions in the general setting of set feasibility.
As a comparison, we make an attempt on connecting the two convergence criteria by linking their key mathematical assumptions to a single physical condition on the \textit{phase diversities} which are the almost only adjustable figures of the phase retrieval problem (see Remark \ref{r:criteria compared}).

The paper is organized as follows.
In the last part of this introductory section, we introduce the mathematical notation used in the paper.
Section \ref{s:problem formulation} is devoted to formulating the phase retrieval problem and addressing in details the key steps towards its solutions using projection algorithms.
A discussion on projection methods for phase retrieval is presented in Section \ref{s:projection algorithm}.
In Section \ref{s:convergence analysis}, convergence results of the DRAP algorithm are established using two different analysis approaches: 1) spectral analysis in Section \ref{subs:TooL_LA} and 2) variational analysis in Section \ref{subs:TooL_VA}.
Numerical simulation is presented in Section \ref{s:numerical experiment}.

\textbf{Mathematical notation.} The underlying space in this paper is a finite dimensional Hilbert space denoted by $\mathcal{H}$.
The element-wise multiplication is denoted by $\odot$.
The element-wise division $\frac{\;\cdot\;}{\;\cdot\;}$, absolute value $|\cdot|$, square $(\,\cdot\,)^2$ and square root $\sqrt{\,\cdot\,}$ operations are also frequently used but without need for extra notation.
$\Re(\cdot)$ and $\Im(\cdot)$ denote the real and the imaginary parts of a complex object in th brackets, respectively.
The imaginary unit is $\jj = \sqrt{-1}$.
$\Id$ denotes the \textit{identity mapping} while $I_n$ denotes the \textit{identity matrix} of size $n$.
The distance to a set $\Omega\subset \mathcal{H}$ is defined by
\begin{equation*}
\dist(\cdot, \Omega) \colon \mathcal{H}\to\Rbb_+\colon x\mapsto \inf_{w\in\Omega}\norm{x-w}
\end{equation*}
and the set-valued mapping
\begin{equation}\label{projection}
\mmap{P_\Omega}{\mathcal{H}}{\Omega}\colon
x\mapsto \set{w\in \Omega}{\norm{x-w}=\dist(x,\Omega)}
\end{equation}
is the \textit{projector} on $\Omega$.
A selection $w\in P_\Omega(x)$ is called a {\em projection} of $x$ on $\Omega$.
The \emph{reflection operator} associated with $\Omega$ is accordingly defined by $R_\Omega:=2P_\Omega-\Id$.
Given a subset $\Omega\subset\mathcal{H}$, the
\emph{Fr\'echet and limiting normal cones} to $\Omega$ at a point $\hat x\in \Omega$ are defined, respectively, as follows:
\begin{gather*}
\fncone{\Omega}(\hat x):= \left\{v\in\mathcal{H}\mid \limsup_{x\TO{\Omega}\hat x,\,x\neq \hat x} \frac{\langle v,x-\hat x \rangle}{\|x-\hat x\|} \leq 0 \right\},
\\
\ncone{\Omega}(\hat x):= \Limsup_{x\TO{\Omega}\hat x}\fncone{\Omega}(x):=\left\{v=
\lim_{k\to\infty}v_k\mid v_k\in \fncone{\Omega}(x_k),\; x_k\TO{\Omega} \hat x \right\},
\end{gather*}
where $x\TO{\Omega} \hat x$ means that $x\to \hat x$ and $x\in\Omega$.
The set of fixed points of an operator $T:\mathcal{H}\rightrightarrows\mathcal{H}$ is defined by
$\Fix T := \{x\in \mathcal{H}\mid x\in T(x)\}$.
Our other basic notation is standard; cf. \cite{Mor06.1,VA}.
$\mathbb{B}_\delta(x)$ stands for the open ball with radius $\delta>0$ and center $x$.
For a linear subspace $V$ of $\mathcal{H}$, 
\[
V^{\perp}:=\set{u\in \mathcal{H}}{\ip{u}{v}=0 \mbox{ for all } v\in V}
\]
is the orthogonal complement subspace of $V$.

\section{Problem formulation}\label{s:problem formulation}

\subsection{Phase retrieval}\label{subs:PR}

Phase diversities and the Fourier transform are key ingredients of the phase retrieval problem studied in this paper.
Recall that adding a phase diversity to the phase of a complex signal is a unitary transform and the (discrete) Fourier transform is also a unitary operator.
Since unitary transforms are one-to-one represented as unitary matrices, the phase retrieval problem can be formulated in the form of matrix-vector-multiplication as follows.
For an unknown complex object $\hat{x} \in \mathbb{C}^n$, let $M \in \mathbb{C}^{N\times n}$ be the propagation matrix which is normalized to be isometric, and $r \in \mathbb{R}_+^N$ be the measured data of $|M\hat{x}|^2$.
The phase retrieval problem is to find an (approximate) solution to the equation:
\begin{equation}\label{phase problem}
r = |Mx|^2 + w,\quad x\in \mathbb{C}^n,
\end{equation}
where $w\in \mathbb{R}^N$ represents unknown noise\footnote{Dimension $n$ corresponds to the pixel totality of one image.}.

\begin{remark}\label{2D versus 1D}
To formulate the phase retrieval problem in the matrix-vector-multiplication form \eqref{phase problem} or any feasibility model in Section \ref{subs:FP}, we need to vectorize all array objects in a consistent manner and rewrite all linear mappings as matrix multiplication operations in higher dimensional spaces, see, for example, \cite[section 2A]{DoeThaVer18}.
This one-to-one conversion allows us to do the theoretical analysis in the simple matrix-vector-multiplication formulation without loss of generality.
\end{remark}

In this paper, we study the phase retrieval setting with several phase diversities, and the propagation matrix $M$ takes the following form:
\begin{equation}\label{matrix M}
M = \frac{1}{\sqrt{m}}\left(
\begin{matrix}
FD_1
\\
FD_2
\\
\cdots
\\
FD_m
\end{matrix}
\right)\in \mathbb{C}^{N\times n},
\end{equation}
where $m\ge 2$ is the number of data images, $F\in \mathbb{C}^{n\times n}$ is the unitary matrix representing the discrete Fourier transform, and $D_d\in \mathbb{C}^{n\times n}$ are unitary matrices representing the \emph{phase diversities} which will be denoted by $\phi_d$ in the sequel $(d=1,2,\ldots,m)$.
Note that $N = mn$.

\begin{remark}[phase modulators versus out-of-focus measurements]\label{r:phase modulator}
There are two widely used techniques of acquiring the PSF images for phase-diversity phase retrieval.
First, a \textit{phase modulator} is used for introducing phase diversities in the pupil plane corresponding to which the images are measured in the focal plane.
Second, the images are registered in out-of-focus planes along the optical axis (i.e, parallel to the focal plane at some known distances) without the use of phase modulator.
It is well known that the two techniques are mathematically equivalent \cite{Goodman05}.
\end{remark}

When a priori knowledge of the solutions is available, that is $\hat{x}\in \chi$ for some known subset $\chi \subset \mathbb{C}^n$, one can expect more accurate phase retrieval.
The formulation \eqref{phase problem} is naturally modified as follows:
\begin{equation}\label{phase problem_chi}
r = |Mx|^2 + w,\quad x\in \chi.
\end{equation}

Following the background developed in \cite{GerSax72,Fie82,BauComLuk02,LukBurLyo02}, we are going to address the problem \eqref{phase problem_chi} using projection algorithms.
The main steps for this solution process will be detailed next.

\subsection{Feasibility models}\label{subs:FP}

Several feasibility models of phase retrieval have been formulated in either the physical domain\footnote{The unknown variable is the signal in the pupil plane.} \cite{LukBurLyo02,LevSta84} or the Fourier domain\footnote{The unknown variable relates to the signal in the pupil plane via the Fourier transform.} \cite{CheFan18}.
Viewing the Fourier transform and phase-diversity addition as unitary transforms, we clarify the relationship between various feasibility models of the phase retrieval problem.

In the physical domain, for each $d=1,2,\ldots,m$, let us denote $r_d$ the measurement of the PSF image $|FD_d(\hat{x})|^2$. 
Define the \textit{intensity constraint} sets as follows \cite{BauComLuk02,LukBurLyo02}:
\begin{equation}\label{amp_con_set}
\Omega_d := \parennn{x\in \mathbb{C}^n \mid (1/m)|FD_d(x)|^2 = r_d}\quad (1\le d\le m).
\end{equation}
Then, the problem \eqref{phase problem_chi} can be approached via the following feasibility problem involving multiple sets:
\begin{equation}\label{PR_physics}
\mbox{find}\quad x\in \bigcap_{d=0}^m \Omega_d,
\end{equation}
where $\Omega_0:=\chi$ captures a priori knowledge of the solutions.

\begin{remark}[nonconvexity feasibility]\label{r:nonconvex}
All the problem models appearing in this paper are nonconvex due to the nonconvexity of the intensity constraints $\Omega_d$ defined in \eqref{amp_con_set}.
\end{remark}

When addressing the phase retrieval problem with noise and model deviation, an appropriate averaging process is essential for suppressing noise.
For this, we consider the following feasibility model in the product space:
\begin{equation}\label{PR_physics_product}
\mbox{find }\quad u \in D\cap \Omega,
\end{equation}
where
\begin{equation}\label{Omega and D}
D:= \parennn{(x,x,\ldots,x) \in \mathbb{C}^{nm} \mid x\in \chi}\; \mbox{ and }\; \Omega:=\Omega_1\times \Omega_2\times \cdots\times \Omega_m.
\end{equation}
The equivalence between \eqref{PR_physics} and \eqref{PR_physics_product} in the general setting of set feasibility finds its root in \cite{Pie84}.
Without a priori constraint, i.e., $\chi=\mathbb{C}^n$, the set $D$ is the ($n$-dimensional subspace) \emph{diagonal} of the product space $\mathbb{C}^{nm}$.
The counterpart of \eqref{PR_physics_product} in the Fourier domain is as follows:
\begin{equation}\label{PR_Fourier}
\mbox{find }\quad y \in  A \cap B,
\end{equation}
where
\begin{equation}\label{A&B}
A:= M(\chi)\mbox{ and } B:=\{y\in \mathbb{C}^N\mid |y|^2=r\}.
\end{equation}

The 2-set feasibility models \eqref{PR_physics_product} and \eqref{PR_Fourier} allow us to adapt various algorithmic schemes including flexible relaxation and regularization for the phase retrieval problem.

The relationships between models \eqref{PR_physics}, \eqref{PR_physics_product} and \eqref{PR_Fourier} in the noiseless setting are as follows.

\begin{proposition}[equivalences of feasibility models] \label{p:problem_relation} Let $\hat{x} \in \mathbb{C}^n$ and $\hat{y}= M\hat{x}$. The following statements are equivalent:
\begin{enumerate}[(i)]
\item $\hat{x}$ is a solution to \eqref{PR_physics};
\item $[x]_m := \underbrace{(\hat{x},\hat{x},\ldots,\hat{x})}_{m~\mathtt{times}}$ is a solution to \eqref{PR_physics_product};
\item $\hat{y}$ is a solution to \eqref{PR_Fourier}.
\end{enumerate}
\end{proposition}
\begin{proof}
The equivalence between $(i)$ and $(ii)$ is widely known \cite{Pie84}, while the equivalence between $(i)$ and $(iii)$ follows from the unitarity property of the matrix $M$ given in \eqref{matrix M}, that is, $M^*M=I_n$.
\qed
\end{proof}
 
\begin{remark}[inconsistent feasibility]\label{r:inconsistent}
In practical circumstances, for example, due to the presence of noise and model deviation, the intersection in \eqref{PR_physics}, \eqref{PR_physics_product} and \eqref{PR_Fourier} is likely to be empty.
There are natural interpretations of inconsistent feasibility in terms of minimization involving indicator and distance functions.
For example, let us interpret the AP method for solving the (possibly inconsistent) feasibility \eqref{PR_Fourier} in terms of classical algorithms for minimization.
The worrisome issue regarding the emptiness of the intersection would be eased when one associates \eqref{PR_Fourier} with the following minimization problem:
\begin{equation}\label{minP_3}
\min_{y\in B}\; f(y):=
\frac{1}{2} \dist^2(y,A).
\end{equation}
In view of Proposition \ref{p:remove_chi} (which is proved later in Section \ref{subs:TooL_LA}), the set $A$ defined in \eqref{A&B} can be assumed to be convex, and hence the objective function $f$ in \eqref{minP_3} is differentiable with the gradient given by $\nabla f(y)=y-P_A(y)$ for every point $y$ \cite{PolRocThi00}.
Then, alternating projection for solving \eqref{PR_Fourier} is precisely the \emph{projected gradient method} for solving \eqref{minP_3}.
\end{remark}

\subsection{Projectors}\label{subs:prox-operator}

The decisive step of solving the feasibility problem \eqref{PR_Fourier} by projection algorithms is to calculate the two projectors on the sets $A$ and $B$ defined in \eqref{A&B}.
Since $B$ is geometrically the product of a number of circles of the complex number plane, an explicit form of the projector $P_B$, which is in general a set-valued mapping, is available \cite{BauComLuk02,LukBurLyo02}:
\begin{align}\label{P_B}
P_B(y) = \sqrt{r} \odot \frac{y}{|y|},\quad \forall y\in \mathbb{C}^N,
\end{align}
with the convention that $\frac{y_i}{|y_i|}=\Sbb$ whenever $y_i=0$, where $\Sbb$ denotes the complex unit circle\footnote{The subscript $i$ indicates the $i$th entry of the object.}.
In numerical computation, the (single-valued) selection of $P_B$ corresponding $\frac{y_i}{|y_i|}=1$ whenever $y_i=0$ is sufficient.

\begin{remark}[projector on regularized sets]\label{r:project on B_eps}
In view of Remark \ref{r:inconsistent}, the set $B$ can have no common point with the set $A$.
For ways of handling such a feasibility gap, one can think of regularizing or approximating the set $B$.
For example, Luke \cite{Luk12} proposed to enlarge the set $B$ to
\begin{equation*}
B_\varepsilon := \left\{y\in \mathbb{C}^N \mid \dist_\phi(y,b)\le \varepsilon,\quad \forall b\in B
\right\},
\end{equation*}
where $\varepsilon\ge 0$ can be viewed as the radius of enlargement, $\dist_\phi$ is the \emph{Bregman distance}, associated with a strictly convex function $\phi:\mathbb{R}^N\to (-\infty,\infty]$ which is differentiable on the interior of its domain, given by
\[
\dist_\phi(y,z) := \phi(|y|) - \phi(|z|) - \ip{\nabla\phi(|y|)}{|y|-|z|},\quad \forall y,z\in \mathbb{C}^N.
\]	
The function $\phi$ should be chosen in accordance with the statistical model of the noise $w$ in \eqref{phase problem_chi}.
More specifically, let us consider the Gaussian and Poisson models of noise, which are perhaps the most relevant to phase retrieval.
The Bregman distance associated with the \emph{half energy kernel operator} $\phi=\frac{1}{2}\|\cdot\|^2$ is simply the Euclidean norm, and it is appropriate for Gaussian noise.
Let us define the function $\phi:\mathbb{R}^N\to (-\infty,\infty]$ by
	\begin{equation}\label{phi for Poisson}
	\phi(v) := \sum_{i=1}^{N} f(v_i),\; \forall v\in \mathbb{R}^N\; \text{where}\; 
	f(t) := \begin{cases}
	t\log t-t & \text{if}\quad t>0,\\
	0  & \text{if}\quad t=0,\\
	\infty  & \text{if} \quad t<0.
	\end{cases}
	\end{equation}
The Bregman distance associated with the function $\phi$ given by \eqref{phi for Poisson} is the \textit{Kullback-Leibler divergence}, and it is appropriate for Poisson noise.
The projector on the regularized set $B_\varepsilon$ can be viewed as an approximation of the projector on $B$, and hence it can be used in the framework of projection methods.
The cyclic projection algorithm using approximate projectors of this type has been analyzed by Luke \cite{Luk12}, and in fact his idea can also be extended to other projection methods. 
However, since the projector on a regularized set is often much more complicated to be computed than the one on the original set, we can instead treat the latter one as an approximation of the former one \cite[page 22]{LukSabTeb19}.
This insight about approximate projectors for inconsistent feasibility allows us to simply use the formula \eqref{P_B} for both analytical and numerical purposes without any worrisome issue.
\end{remark}

The projector on the set $A$ can also be explicitly described\footnote{Note that convexity of $A$ is not required in Lemma \ref{l:P_A}.}.
We make use of the following notation:
\begin{equation*}
[\chi]_m := \left\{[x]_m \mid x\in \chi \right\}\; \mbox{ where } \; [x]_m := \underbrace{(x,x,\ldots,x)}_{m~ \mathtt{times}}.
\end{equation*}

\begin{lemma}\label{l:P_A} For the propagation matrix $M$ given in \eqref{matrix M}, it holds that
	\begin{align}\label{P_A}
	P_A(y) = MP_{\chi}\paren{M^*y},\quad \forall y\in \mathbb{C}^N.
	\end{align}
\end{lemma}
\begin{proof}
	Let us first define the unitary matrix based on the matrix $M$ as follows:
	\begin{equation}\label{matrix U}
	U := \left(
	\begin{matrix}
	FD_1 & 0 & \cdots & 0
	\\
	0 & FD_2 & \cdots & 0
	\\
	\cdots & \cdots & \cdots & \cdots
	\\
	0 & 0 & 0 & FD_m
	\end{matrix}
	\right) \in \mathbb{C}^{N\times N}.
	\end{equation}
	This block diagonal matrix is unitary since all of its constituent blocks are so.
	By the structure of $M$ and $U$, we have that
	\[
	A = M(\chi) = \frac{1}{\sqrt{m}}U\paren{[\chi]_m}.
	\]
	Since $U$ is unitary, it holds that
	\begin{align}\label{pass projection}
	P_A(y) = P_{\frac{1}{\sqrt{m}}U\paren{[\chi]_m}}(y)
	= P_{U\paren{\frac{1}{\sqrt{m}}[\chi]_m}}(y)
	= U\paren{P_{\frac{1}{\sqrt{m}}[\chi]_m}(U^*y)}.
	\end{align}
	Since $\parenn{\mathbb{C}^n}_m$ is a subspace containing $\frac{1}{\sqrt{m}}\parenn{\chi}_m$, by the properties of the metric projection, we have that
	\begin{align}\label{split projections}
	P_{\frac{1}{\sqrt{m}}[\chi]_m} = P_{\frac{1}{\sqrt{m}}[\chi]_m}\circ P_{\parenn{\mathbb{C}^n}_m}.
	\end{align}
	We next calculate $U^*y$.
	Note that $U^*$ is also a block diagonal matrix whose blocks are the conjugate transpose of the corresponding blocks of $U$.
	Let us denote $c_k$ the column vector whose entries taken from $y$ correspond to the block $(FD_k)^*$ of $U^*$, $(1\le k\le m)$.
	We have that
	\begin{equation*}
	U^*y = \left(
	\begin{matrix}
	(FD_1)^*c_1
	\\
	(FD_2)^*c_2
	\\
	\ldots
	\\
	(FD_m)^*c_m
	\end{matrix}
	\right).
	\end{equation*}
	Since $\parenn{\mathbb{C}^n}_m$ is the $n$-dimensional diagonal of the product space $\mathbb{C}^{nm}$, we obtain by solving the minimizing problem \eqref{projection} that
	\begin{equation}\label{proj of U^*y}
	P_{\parenn{\mathbb{C}^n}_m}(U^*y) = \frac{1}{m}\left[\sum_{k=1}^m (FD_k)^*c_k\right]_m = \frac{1}{\sqrt{m}}\left[M^*y\right]_m.
	\end{equation}
	Plugging \eqref{proj of U^*y} and \eqref{split projections} into \eqref{pass projection} yields that
	\begin{align*}
	P_A(y) &= U\paren{P_{\frac{1}{\sqrt{m}}[\chi]_m}\frac{1}{\sqrt{m}}\left[M^*y\right]_m} =
	U\paren{\frac{1}{\sqrt{m}}P_{[\chi]_m}\left[M^*y\right]_m}
	\\
	&= \frac{1}{\sqrt{m}}U\paren{P_{[\chi]_m}\left[M^*y\right]_m} = \frac{1}{\sqrt{m}}U\paren{\left[P_{\chi}\paren{M^*y}\right]_m}
	= MP_{\chi}\paren{M^*y}.
	\end{align*}
	The proof is complete.
\qed
\end{proof}

The formula \eqref{P_A} shows that the complexity of $P_A$ heavily depends on that of $P_{\chi}$.

\section{Projection algorithms}\label{s:projection algorithm}

Projection algorithms for phase retrieval can be considered as descendants of the well known Gerchberg--Saxton (GS) algorithm \cite{GerSax72} which deals with phase retrieval given the amplitude constraint and a single PSF image. 
Their introduction has been motivated by the rapidly growing application of phase retrieval originated from a wide variety of physical settings.
For example, the famous \emph{input-output}, \emph{output-output} and \emph{hybrid-input-output} algorithms \cite{Fie82} arose up when dealing with the support and the real and nonnegative constraints instead of the amplitude constraint as the GS method.
Extensions for solving problems given multiple images and for obtaining better restoration have been among the main objectives of this class of phase retrieval algorithms.
In light of the groundwork \cite{BauComLuk02}, in Section \ref{subs:FP} we have interpreted the phase retrieval problem \eqref{phase problem_chi} as a feasibility problem in one of the equivalent forms \eqref{PR_physics}, \eqref{PR_physics_product} and \eqref{PR_Fourier}. 
Having calculated the projectors $P_A$ and $P_B$ in Section \ref{subs:prox-operator}, we are now ready to discuss algorithmic schemes for the solutions. 
From now on, we analyze the feasibility model \eqref{PR_Fourier}.

AP and DR are perhaps the most widely known solution methods for feasibility and have been the basis for a wide variety of modification and regularization schemes.
We refer the reader to, for example,  \cite{KruLukNgu18,BauMou17} for an overview of these basic methods in the setting of set feasibility. 
For an early discussion in the context of phase retrieval, we refer the reader to the surveys \cite{LukBurLyo02,BauComLuk02}.
It has been observed that AP is stable, always convergent and to some extent able to suppress noise, but the convergence speed can be very slow \cite{Fie82}.
In contrast, DR can be fast in convergence, but sensitive to noise and model deviation \cite{Luk08}. 
Indeed, only relaxations of DR can be used for problems in the presence of noise and model mismatch.

The use of the Krasnoselski--Mann relaxation is perhaps the most widely known.
Mathematically, it is the convex combination of the DR operator $T_{\mathtt{DR}}:=\frac{1}{2}\paren{R_AR_B+\Id}$ and the identity mapping:
\begin{equation*}
T_{\mathtt{KMDR}} := \beta T_{\mathtt{DR}} + (1-\beta)\Id,
\end{equation*}
where $\beta \in (0,1]$ is the relaxation parameter. 
The Fienup's \emph{hybrid-input-output} (HIO) method \cite{Fie82} can be viewed as a relaxation of DR:
\begin{equation*}
T_{\mathtt{HIO}} := P_A\left((1+\beta)P_B-\Id\right) - \left(\beta P_B-\Id\right),
\end{equation*}
where $\beta\in (0,1]$ is the relaxation parameter.
Another relaxation of DR known as the \emph{relaxed averaged alternating reflections} (RAAR) algorithm was proposed and analyzed in \cite{Luk05,Luk08} for phase retrieval. 
It is the convex combination of the DR operator and one of the projectors:
\begin{equation*}
T_{\mathtt{RAAR}} := \beta T_{\mathtt{DR}} + (1-\beta)P_B,
\end{equation*}
where $\beta\in (0,1]$ is the relaxation parameter.
Inexact versions of RAAR were also proposed and analyzed in \cite{Luk08}.
The DRAP algorithm \cite{Tha18} is another relaxation of DR:
\begin{equation}\label{DRAP}
T_{\mathtt{DRAP}} := P_A\left((1+\lambda)P_B-\lambda\Id\right)-\lambda\left(P_B-\Id\right),
\end{equation}
where $\lambda\in [0,1]$ is the relaxation parameter\footnote{Relaxation parameter zero is not allowed for KMDR, HIO and RAAR.}.

Interestingly, in the phase retrieval setting \eqref{PR_Fourier}, $T_{\mathtt{DRAP}}$ coincides with the convex combination of the AP and DR operators provided that $\chi$ is an affine set.
The latter condition implies that the set $A=M(\chi)$ given by \eqref{A&B} is affine.
Hence, the projector $P_A$ is linear and we obtain the the following expression:
\begin{align}\nonumber
T_{\mathtt{DRAP}} &=
P_A\left((1-\lambda)P_B+\lambda(2P_B-\Id)\right)-\lambda\left(P_B-\Id\right)
\\\nonumber
&= \lambda\paren{\Id + P_A(2P_B-\Id) - P_B} + (1-\lambda)P_AP_B 
\\\label{DRAP_0}
&= \lambda T_{\mathtt{DR}} + (1-\lambda)T_{\mathtt{AP}},
\end{align}
where $T_{\mathtt{AP}}:=P_AP_B$ is the AP operator.

\begin{remark}\label{r:2 expressions}
The two expressions \eqref{DRAP} and \eqref{DRAP_0} play their own role in explaining interesting features of DRAP\footnote{They do differ in general settings.}.
On the one hand, only two projections are required for computing an iteration of \eqref{DRAP} ($P_B$ once and $P_A$ once) compared to three projections for \eqref{DRAP_0} ($P_B$ once and $P_A$ twice).
This means that the computational complexity of DRAP is at the same level as that of the other projection methods if \eqref{DRAP} is used in numerical implementation.
On the other hand, the expression \eqref{DRAP_0} as a convex combination of $T_{\mathtt{AP}}$ and $T_{\mathtt{DR}}$ explains better the idea leading to the introduction of DRAP as a relaxation of DR compared to the less intuitive form \eqref{DRAP}.
\end{remark}

Plugging the two projectors \eqref{P_B} and \eqref{P_A} into \eqref{DRAP}, we come up with the following explicit form of DRAP for addressing the feasibility problem \eqref{PR_Fourier}:
\begin{equation}\label{DRAP_chi}
\begin{aligned}
y^+ &\in T_{\mathtt{DRAP}}(y)
\\
&= MP_{\chi}\Big{(}M^*\Big{(}(1+\lambda) \sqrt{r} \odot\frac{y}{|y|}-\lambda y\Big{)}\Big{)} - \lambda\Big{(}\sqrt{r}\odot\frac{y}{|y|}-y\Big{)},
\end{aligned}
\end{equation}
where $y$ and $y^+$ stand for the two consecutive iterations of DRAP.
In the case $\chi=\mathbb{C}^n$, \eqref{DRAP_chi} further reduces to
\begin{equation}\label{DRAP_without_chi}
\begin{aligned}
T_{\mathtt{DRAP}}(y) &= MM^*\Big{(}(1+\lambda)\sqrt{r}\odot\frac{y}{|y|}-\lambda y\Big{)} - \lambda\Big{(}\sqrt{r} \odot\frac{y}{|y|}-y\Big{)}
\\
&= \lambda\paren{I_N-MM^*}(y) + \paren{(1+\lambda)MM^*-\lambda I_N}\Big{(}\sqrt{r}\odot\frac{y}{|y|}\Big{)}.
\end{aligned}
\end{equation}

In the remainder of this paper, we analyze the DRAP algorithm in the phase retrieval setting \eqref{PR_Fourier} and demonstrate its advantages over the other algorithms.

\section{Convergence analysis}\label{s:convergence analysis}

In this section, we study convergence properties of DRAP using two different analysis schemes. 
Since the problem \eqref{PR_Fourier} is nonconvex, we can only obtain local convergence criteria though it is observed from numerical results that the quality of phase retrieval is not affected by the starting point for the algorithm.

\subsection{A result from spectral analysis}\label{subs:TooL_LA}

The analysis in this section is based on the observation that the projector $P_A$ given by \eqref{P_A} is linear, and the projector $P_B$ given by \eqref{P_B} also has a good first order approximation around any solution of \eqref{PR_Fourier}.
We follow the analysis approach initiated by Chen and Fannjiang \cite{CheFan18} where they established a local linear convergence result for the DR algorithm.
The mentioned result of \cite{CheFan18} was later extended for the RAAR algorithm in \cite{LiZho17}.
We will show that DRAP also enjoys that kind of convergence result\footnote{Similar results for the HIO algorithm are unknown.}.
In this section, we assume that the lowest intensity of the images is strictly positive:
\begin{equation}\label{vanish_nowhere}
\min_{1\le i\le N} r_i > 0.
\end{equation}

\begin{remark}\label{r:about vanish nowhere}
When the phase diversities $\phi_d$ are assumed to be continuous random variables, condition \eqref{vanish_nowhere} is satisfied almost surely \cite{CheFan18}.
\end{remark}

We first analyze DRAP in the form \eqref{DRAP_without_chi} for solving \eqref{PR_Fourier} with $\chi=\mathbb{C}^n$.
Let us denote
\begin{equation*}
{Y}:= \diag \paren{\frac{\hat{y}}{|\hat{y}|}} \in \mathbb{C}^{N\times N},\; {L} := {Y}^*M \in \mathbb{C}^{N\times n},
\end{equation*}
where $\hat{y}$ is a solution to \eqref{PR_Fourier} and $\diag(\cdot)$ denotes the diagonal matrix with elements on its diagonal taken from the vector in the brackets.
Since $r=|\hat{y}|^2$ vanishes nowhere\footnote{Recall that the square amplitude is element-wise.} by \eqref{vanish_nowhere}, for all $y$ sufficiently close to $\hat{y}$, $|y|$ also vanishes nowhere.
In particular, for a fixed vector $v\in \mathbb{C}^N$, the vector $|\hat{y} + \varepsilon v|$ vanishes nowhere provided that $\varepsilon$ is sufficiently small. The next lemma establishes the first order approximation of $T_{\mathtt{DRAP}}$ as a complex vector valued function around $\hat{y}$ in a given direction.

\begin{lemma}[first order approximation of $T_{\mathtt{DRAP}}$]\label{l:gradient map}
For a vector $v\in \mathbb{C}^N$ and a sufficiently small number $\varepsilon>0$, we have
\begin{equation}\label{gradient formula}
T_{\mathtt{DRAP}}(\hat{y}+\varepsilon v) - T_{\mathtt{DRAP}}(\hat{y}) = \varepsilon\,Y\, \nabla(\mu) + o(\varepsilon),
\end{equation}
where $\mu := Y^*v\;$ and
$\;\nabla(\mu) := \lambda(I_N-LL^*)\mu + {\rm j} ((1+\lambda)LL^*-\lambda I_N)\Im(\mu)$.
\end{lemma}

\begin{proof}
Let us first denote
\begin{equation*}
w_{\varepsilon} := \frac{\hat{y}+\varepsilon v}{|\hat{y}+\varepsilon v|}\; \mbox{ and }\; Y_{\varepsilon} := \diag(w_{\varepsilon}).
\end{equation*}
In view of \eqref{DRAP_without_chi}, we have that
\begin{align*}
T_{\mathtt{DRAP}}(\hat{y}) &= \hat{y} = \paren{(1+\lambda)MM^*-\lambda I_N}Y r,
\\
T_{\mathtt{DRAP}}(\hat{y}+\varepsilon v) &= \lambda\paren{I_N - MM^*}(\hat{y}+\varepsilon v) + \paren{(1+\lambda)MM^*-\lambda I_N}Y_{\varepsilon}r
\\
&= \varepsilon\lambda\paren{I_N - MM^*}(v) + \paren{(1+\lambda)MM^*-\lambda I_N}Y_{\varepsilon}r.
\end{align*}
Then
\begin{equation}\label{gradient estimate difference}
\begin{aligned}
T_{\mathtt{DRAP}}(\hat{y}+\varepsilon v) - T_{\mathtt{DRAP}}(\hat{y})
& = \varepsilon\lambda\paren{I_N - MM^*}v 
\\
&+ \paren{(1+\lambda)MM^*-\lambda I_N}(Y_{\varepsilon}-Y)r.
\end{aligned}
\end{equation}
The following formula for the first order approximation of $(Y_{\varepsilon}-Y)r$ can be calculated directly:
\begin{equation}\label{first order approximation}
(Y_{\varepsilon}-Y)r = \varepsilon \jj Y \Im\paren{Y^*v} + o(\varepsilon).
\end{equation}
Substituting \eqref{first order approximation} into \eqref{gradient estimate difference} yields
\begin{align*}
T_{\mathtt{DRAP}}(\hat{y} +& \varepsilon v) - T_{\mathtt{DRAP}}(\hat{y})
\\
=\; & \varepsilon\lambda\paren{I_N-MM^*}v + \varepsilon \jj\paren{(1+\lambda)MM^*-\lambda I_N} Y\Im \paren{Y^*v} + o(\varepsilon)
\\
=\; & \varepsilon\lambda Y \paren{I_N-LL^*}\mu + \varepsilon \jj Y\paren{(1+\lambda)LL^*-\lambda I_N}\Im \paren{\mu} + o(\varepsilon).
\end{align*}
The proof is complete.
\qed
\end{proof}

The next step is to analyze the spectrum of the real decomposition of the complex matrix $L$ as follows:
\begin{equation*}
\mathcal{L} :=
\paren{\begin{matrix}
	\Re(L) &
	-\Im(L)
	\end{matrix}
}\; \in \mathbb{R}^{N\times 2n}.
\end{equation*}
Note that $\mathcal{L}$ is isometric since $L$ is so.
Define also the real decomposition of a complex vector by
\begin{equation*}
G(x) := \paren{\begin{matrix}
	\Re(x)
	\\
	\Im(x)
	\end{matrix}
}\; \in \mathbb{R}^{2n},\quad \forall x \in \mathbb{C}^{n}.
\end{equation*}
Let $1 \ge \sigma_1\ge \sigma_2 \cdots \ge \sigma_{2n}\ge \sigma_{2n+1}=\cdots = \sigma_N=0$ be the singular values of $\mathcal{L}$ with the corresponding right singular vectors $\parennn{v_k\in \mathbb{R}^{2n} : k=1,\ldots,2n}$ and the left singular vectors $\parennn{u_k\in \mathbb{R}^N : k=1,\ldots,N}$.
We have by the definition of the singular value decomposition (SVD) that
\begin{gather*}
\Re \paren{LG^{-1}(v_k)}
=  \mathcal{L}v_k = \sigma_k u_k,
\\
\sigma_k G^{-1}(v_k) = G^{-1}(\sigma_k v_k) = G^{-1}\paren{\mathcal{L}^Tu_k} = G^{-1}\paren{
	\begin{matrix}
	\Re\paren{L^T}u_k
	\\
	-\Im\paren{L^T}u_k
	\end{matrix}
}
= L^*u_k.
\end{gather*}

The next technical result regarding the spectrum of $\mathcal{L}$ is crucial.

\begin{lemma}\label{l:first spectrum}\cite[Proposition 5.6]{CheFan18}
	There holds that $v_1=G(\hat{x})$, $v_{2n}=G(-{\rm j}\hat{x})$, $\sigma_1=1$, $\sigma_{2n}=0$ and $u_1=|\hat{y}|$.
\end{lemma}

Thanks to Lemma \ref{l:first spectrum} and the definition of the SVD, one has the following expression of the second largest singular value of $\mathcal{L}$:
\begin{equation}\label{second singular value}
\begin{aligned}
\sigma_2 &= \max\parennn{\norm{\mathcal{L}^T u}:u\in \mathbb{R}^{N}, u \perp u_1,\|u\|=1}
\\
&= \max\parennn{\norm{\mathcal{L}v}:v\in \mathbb{R}^{2n}, v \perp v_1,\|v\|=1}
\\
&= \max\parennn{\norm{\Im(Lx)}:x\in \mathbb{C}^n, x\; \bot \; \jj \hat{x},\|x\|=1}.
\end{aligned}
\end{equation}

The following theorem establishes linear convergence of the DRAP algorithm for solving \eqref{PR_Fourier}.
Since phase retrieval is ambiguous (at least) up to a global phase shift\footnote{That is the first element of the orthogonal basis of Zernike polynomials.}, the following distance between two complex vectors is of interest:
\begin{equation}\label{dist up to phase shift}
\dist_{\rm opt}(x,u) := \min_{\alpha\in \mathbb{C},|\alpha|=1}\|\alpha x-u\|,\quad \forall x,u\in \mathcal{H}.
\end{equation}

\begin{theorem}[linear convergence of DRAP]\label{t:LA}
In the setting of \eqref{PR_Fourier} with $\chi=\mathbb{C}^n$, suppose that
\begin{align}\label{spectral gap assumption}
\sigma_2 := \max\parennn{\|\Im (Lx)\|:x\in \mathbb{C}^n,\; \|x\| =1,\; x\; \bot\; {\rm j}\hat{x}} <1.
\end{align}
Let $y^{(k+1)} \in T_{\mathtt{DRAP}}\paren{y^{(k)}}$ be a sequence generated by $T_{\mathtt{DRAP}}$ in the form of \eqref{DRAP_without_chi} with $y^{(0)} = Mx^{(0)}$ for some $x^{(0)}\in \mathbb{C}^n$.
If $x^{(0)}$ is sufficiently close to $\hat{x}$, then there exists a number $c\in (\sigma_2,1)$ such that
\begin{equation*}
\dist_{\rm opt}(x^{(k)},\hat{x}) \le c^{k}\dist_{\rm opt}(x^{(0)},\hat{x}),\quad  \paren{\forall k\in\mathbb{N}}
\end{equation*}
where $x^{(k)}:=M^*y^{(k)}$ $(k=1,2,\ldots)$.
\end{theorem}

\begin{proof}
First, the optimal global phase shift defined by \eqref{dist up to phase shift} is given by \cite{LiZho17}:
\begin{align}\notag
\alpha^{(k)} &= \argmin_{\alpha}\parennn{\|\alpha x^{(k)}-\hat{x}\|:|\alpha|=1,\;\alpha\in\mathbb{C}}
\\
\label{optimal phase at iterate}
&= x^{(k)*}\hat{x}/\left|{x^{(k)}}^*\hat{x}\right| = y^{(k)*}\hat{y}/\left|{y^{(k)}}^*\hat{y}\right|.
\end{align}
Let us denote $\eta^{(k)} :=Y^*(\alpha^{(k)}y^{(k)}-\hat{y})$.
Thanks to Lemma \ref{l:gradient map}, we have that
\begin{align*}
Y^*\paren{\alpha^{(k)} y^{(k+1)}-\hat{y}} &=  Y^*\paren{\alpha^{(k)} T_{\mathtt{DRAP}}\paren{y^{(k)}} - T_{\mathtt{DRAP}}\paren{\hat{y}}}
\\
&=  Y^*\paren{T_{\mathtt{DRAP}}\paren{\alpha^{(k)}y^{(k)}} - T_{\mathtt{DRAP}}\paren{\hat{y}}}
\\
&=  Y^*\paren{Y \nabla\paren{Y^*(\alpha^{(k)}y^{(k)}-\hat{y})}} + o(\|\alpha^{(k)}y^{(k)}-\hat{y}\|)
\\
&= \nabla({\eta^{(k)}}) + o(\|{\eta^{(k)}}\|).
\end{align*}
Multiplying both sides of the above equality by $L^*=M^*Y$ and taking the isometry property of $L$ into account, we obtain that
\begin{equation}\label{gap in x}
\begin{aligned}
&\alpha^{(k)}{x^{(k+1)}}-\hat{x}
= L^*Y^*\paren{\alpha^{(k)}y^{(k+1)} - \hat{y}}
= L^*\nabla({\eta^{(k)}}) + o(\|{\eta^{(k)}}\|)
\\
=\; &\lambda L^*(I_N-LL^*){\eta^{(k)}} + \jj L^*\paren{(1+\lambda)LL^* - \lambda I_N}\Im({\eta^{(k)}}) + o(\|{\eta^{(k)}}\|)
\\
=\; &\jj L^*\Im({\eta^{(k)}}) + o(\|{\eta^{(k)}}\|).
\end{aligned}
\end{equation}
Due to \eqref{optimal phase at iterate} and the fact that $\ip{|\hat{y}|}{\jj|\hat{y}|}=0$ we have
\begin{align*}
\ip{{\eta^{(k)}}}{\jj|\hat{y}|} &= \ip{\frac{\hat{y}^*}{|\hat{y}|}\odot(\alpha^{(k)}y^{(k)}-\hat{y})}{\jj|\hat{y}|}
\\
&= \ip{\frac{\hat{y}^*}{|\hat{y}|}\odot \alpha^{(k)}y^{(k)}}{\jj|\hat{y}|} + \ip{|\hat{y}|}{\jj|\hat{y}|}
\\
&= \ip{\alpha^{(k)}\frac{{\hat{y}^*}}{|\hat{y}|}\odot y^{(k)}}{\jj|\hat{y}|}
\\
&= \ip{\frac{{y^{(k)}}^*\hat{y}}{\left|{y^{(k)}}^*\hat{y}\right|}\paren{\frac{{\hat{y}^*}}{|\hat{y}|}\odot y^{(k)}}}{\jj|\hat{y}|}
\\
&= \ip{\left|{y^{(k)}}^*\hat{y}\right|\frac{1}{|\hat{y}|}}{\jj|\hat{y}|} = 0.
\end{align*}
In other words, $\eta^{(k)} \perp \jj|\hat{y}|$. By basic properties of the Hermitian inner product, one has $\Re\paren{\eta^{(k)}} \perp \jj|\hat{y}|$. As a result, $\Im({\eta^{(k)}}) \perp |\hat{y}|$.
Taking Lemma \ref{l:first spectrum} into account, we have just shown that $\Im({\eta^{(k)}})$ is orthogonal to $u_1=|\hat{y}|$ which is the first left singular vector of $\mathcal{L}$.
This together with the expression \eqref{second singular value} of $\sigma_2$ implies that
\begin{align}\label{sigma_2 estimate}
\norm{\mathcal{L}^T\Im({\eta^{(k)}})} \le \sigma_2\norm{\Im({\eta^{(k)}})}.
\end{align}
Combining \eqref{dist up to phase shift}, \eqref{gap in x} and \eqref{sigma_2 estimate} yields that
\begin{equation}\label{dist new}
\begin{aligned}
\dist_{\rm opt}({x^{(k+1)}},\hat{x}) &= \min_{\alpha\in \mathbb{C},|\alpha|=1} \norm{\alpha {x^{(k+1)}}-\hat{x}}
\\
&\le \norm{\alpha^{(k)}{x^{(k+1)}}-\hat{x}}
\\
&= \norm{L^*\Im({\eta^{(k)}})} + o(\|{\eta^{(k)}}\|)
\\
&= \norm{\mathcal{L}^T\Im({\eta^{(k)}})} + o(\|{\eta^{(k)}}\|)
\\
&\le \sigma_2\norm{\Im({\eta^{(k)}})} + o(\|{\eta^{(k)}}\|)
\\
&\le \sigma_2\norm{\eta^{(k)}} + o(\|{\eta^{(k)}}\|).
\end{aligned}
\end{equation}
Since $\sigma_2<1$ by assumption \eqref{spectral gap assumption}, there exists a number $c\in (\sigma_2,1)$ such that for all ${\eta^{(k)}}$ with $\|{\eta^{(k)}}\|$ sufficiently small, it holds that
\begin{align}\label{sigma_2 < c}
\sigma_2\norm{\eta^{(k)}} + o(\|{\eta^{(k)}}\|) \le c\norm{{\eta^{(k)}}}.
\end{align}
Combining \eqref{dist new}, \eqref{sigma_2 < c} and the definition of $\eta^{(k)}$ yields
\begin{align*}
\dist_{\rm opt}({x^{(k+1)}},\hat{x}) \le c\norm{{\eta^{(k)}}} = c\dist_{\rm opt}({x^{(k)}},\hat{x}),\quad (k=1,2,\ldots).
\end{align*}
The proof is complete.
\qed
\end{proof}

\begin{remark}
In view of \cite[Proposition 6.2]{CheFan18}, the assumption \eqref{spectral gap assumption} of Theorem \ref{t:LA} is satisfied almost surely.
\end{remark}

\begin{remark}[region of convergence]\label{r:area of convergence}
Since the algorithm operates in the underlying space $\mathbb{C}^N$, for the sake of brevity, let us speak of region around $\hat{y}=M(\hat{x})$ instead of $\hat{x}$.
In view of Theorem \ref{t:LA}, such a convergence region, if exists, is mutually dependent on the constant $c$.
More specifically, given a number $c\in (\sigma_2,1)$, it is the region in which the first order approximation \eqref{gradient formula} of $T_{\mathtt{DRAP}}$ around $\hat{y}$ is valid and condition \eqref{sigma_2 < c} is satisfied for all $k\in\mathbb{N}$.
Note that the latter involves not only $\sigma_2$ and $c$ but also the sequence $y^{(k)}$ itself.
The intersection of the regions over all possible sequences complied with $T_{\mathtt{DRAP}}$ can be taken as the region of convergence.
Obviously, such a statement is not informative and hence it has not ever been an objective of local convergence analysis.
\end{remark}

\begin{remark}[influence of $\lambda$ on convergence]\label{lambda vs convergence}
In view of Theorem \ref{t:LA}, the relaxation parameter $\lambda$ obviously has influence on the region in which the first order approximation of $T_{\mathtt{DRAP}}$ (Lemma \ref{l:gradient map}) is valid and condition \eqref{sigma_2 < c} is satisfied, however, its influences on the convergence speed of DRAP is unclear\footnote{Of course, the influence is clearly observed from numerical computation.}.
\end{remark}

We have analyzed the DRAP algorithm in the phase retrieval setting \eqref{PR_Fourier} with $\chi=\mathbb{C}^n$.
The latter condition limits the effectiveness of Theorem \ref{t:LA} to phase retrieval without a priori constraint.
In the remainder of this section, we will show that the convergence criterion can also be applicable to phase retrieval problems with an \textit{amplitude constraint}, which is a helpful prior information and often available in practice\footnote{This is because the light distribution in the pupil plane is often known, for example, it can be uniform or truncated Gaussian.}.

The amplitude constraint is described by
\begin{equation}\label{amp constraint}
\chi=\parennn{x\in \mathbb{C}^n\mid |x|=a},
\end{equation}
where $a\in \mathbb{R}_+^n$ is the known amplitude of the complex signal.
The next result shows that the problem \eqref{PR_Fourier} with an amplitude constraint can equivalently be reformulated as a problem without a priori constraint in a higher dimensional space.

\begin{proposition}\label{p:remove_chi}
The problem \eqref{PR_Fourier} with the amplitude constraint \eqref{amp constraint} can equivalently be reformulated as: 
\begin{equation}\label{PR_Fourier'}
\mbox{find }\quad \mathbf{y} \in  \mathbf{A} \cap \mathbf{B},
\end{equation}
where $\mathbf{A}:= \mathbf{M}(\mathbb{C}^n),\; \mathbf{B}:=\{\mathbf{y}\in \mathbb{C}^{N+n}\mid |\mathbf{y}|^2=\mathbf{r}\}  \subset \mathbb{C}^{N+n}$ with
\begin{equation}\label{matrix M'}
\mathbf{M} := \frac{1}{\sqrt{m+1}}\left(
\begin{matrix}
\sqrt{m}M
\\
I_n
\end{matrix}
\right) \in \mathbb{C}^{(N+n)\times n} \mbox{ and }\;
\mathbf{r} := \frac{1}{m+1}\left(
\begin{matrix}
m r
\\
a^2
\end{matrix}
\right) \in \mathbb{R}_+^{N+n}.
\end{equation}
\end{proposition}

\begin{proof}
For convenience, let us recall that $N=nm$ according to \eqref{matrix M}.
We first observe that $M$ is isometric if and only if $\mathbf{M}$ is isometric since
\[
\mathbf{M}^*\mathbf{M} = \frac{m}{m+1}M^*M + \frac{1}{m+1}I_n.
\]

Let $\hat{y}\in\mathbb{C}^N$ be a solution to \eqref{PR_Fourier}.
That is, $\hat{y}=M\hat{x}$ with $|\hat{x}|=a$ and $|\hat{y}|^2=r$.
Define $\hat{\mathbf{y}} := \frac{1}{\sqrt{m+1}}\left(
\begin{matrix}
\sqrt{m}\hat{y}
\\
\hat{x}
\end{matrix}
\right) \in\mathbb{C}^{N+n}$.
Then $\hat{\mathbf{y}}=\mathbf{M}\hat{x} \in \mathbf{A}$ and $|\hat{\mathbf{y}}|^2=\mathbf{r}$.
This means that $\hat{\mathbf{y}}$ is a solution to \eqref{PR_Fourier'}.

Conversely, let $\hat{\mathbf{y}}\in\mathbb{C}^{N+n}$ be a solution to \eqref{PR_Fourier'}.
That is, $\hat{\mathbf{y}}=\mathbf{M}\hat{x}$ with $\hat{x}\in\mathbb{C}^n$ and $|\hat{\mathbf{y}}|^2=\mathbf{r}$.
By the definition of $\mathbf{M}$ and $\mathbf{r}$ in \eqref{matrix M'}, we have $|\hat{x}|=a$, or equivalently, $\hat{x}\in\chi$.
Define $\hat{y}=M\hat{x} \in \mathbb{C}^N$.
Then $\hat{y}\in M(\chi) = A$.
We have also from \eqref{matrix M'} that $|\hat{y}|^2=r$.
This means that $\hat{y}$ is a solution to \eqref{PR_Fourier}.

The proof is complete.
\qed
\end{proof}

\begin{remark}\label{r:cover amplitude constrain}
Proposition \ref{p:remove_chi} shows that the convergence criterion formulated in Theorem \ref{t:LA} is indeed applicable to not only phase retrieval problems without a priori constraint but also those involving an amplitude constraint.
Compared to the earlier convergence results for the DR algorithm in \cite[Theorem 5.1]{CheFan18} and the RAAR algorithm \cite[Theorem 3]{LiZho17}, this new observation widens the applicable scope of this type of convergence results.
\end{remark}

\subsection{A result from variational analysis} \label{subs:TooL_VA}

We recall a number of mathematical notions needed for formulating a local linear convergence criterion for the DRAP algorithm using the analysis scheme of \cite{LukNguTam18} and discuss the validity of the imposed assumptions in the setting of phase retrieval.

\begin{definition}[prox-regularity of sets] \cite{PolRocThi00}
	A set $\Omega$ is called \emph{prox-regular} at a point $\hat{y}\in \Omega$ if the projector $P_{\Omega}$ is single-valued around $\hat{y}$. 
\end{definition}

Prominent example of prox-regularity is that a closed and convex set is prox-regular at every of its points. 
In particular, the set $A=M(\chi)$ in \eqref{A&B} has this property whenever $\chi$ is convex.
The next statement finds its root in the original work \cite[Section 3.1]{Luk08}.

\begin{lemma}[prox-regularity of $B$]\label{l:prox_B}  \cite[Lemma 6.2$(i)$]{NguLukSolVer20} The set $B$ defined in \eqref{A&B} is prox-regular at every of its points.
\end{lemma}

\begin{definition}[pointwise almost averaged operators]
\label{d:p.a.a.}
\cite[Definition 2.2 and Proposition 2.1]{LukNguTam18}
A (not necessarily nonexpansive) fixed point operator $T:\mathcal{H}\rightrightarrows \mathcal{H}$ is called \emph{pointwise almost averaged} on a set $\Omega$ at a point $y\in \Omega$ with violation $\varepsilon$ and averaging constant $\alpha>0$ if for all $z\in U$, $z^+\in T(z)$ and $y^+\in T(y)$,
\begin{align*}
\norm{z^+- y^+}^2 \le \paren{1+\varepsilon}\norm{z-y}^2 - \frac{1-\alpha}{\alpha}\norm{(z^+-z)-(y^+-y)}^2.
\end{align*}
$T$ is called \emph{almost averaged} on $\Omega$ with violation $\varepsilon$ and averaging constant $\alpha$ if it is pointwise almost averaged on $\Omega$ at every point $y\in\Omega$ with that violation and averaging constant.
When the violation $\varepsilon$ is zero, the quantifier `almost' is dropped.
\end{definition}

For the meaning of the quantifiers `pointwise' and `almost' appearing in Definition \ref{d:p.a.a.} as well as the motivation of the property, we refer the reader to the original work on pointwise almost averaged operators \cite{LukNguTam18}.
The following statement claims this property for $T_{\mathtt{DRAP}}$ as a fixed point operator.

\begin{lemma}[almost averagedness of $T_{\mathtt{DRAP}}$]\label{l:p.a.a.} Let $\hat{y}\in \mathbb{C}^N$ be a solution to \eqref{PR_Fourier}. Then for any $\varepsilon>0$ arbitrarily small, there exist numbers $\delta>0$ and $\alpha\in (0,1)$ dependent on $\varepsilon$ such that $T_{\mathtt{DRAP}}$ is almost averaged on $\Ball_{\delta}(\hat{y})$ with violation $\varepsilon$ and averaging constant $\alpha$.
\end{lemma}

\begin{proof}
Let $\varepsilon>0$ be a positive number, which can be arbitrarily small. Since the set $B$ is prox-regular at $\hat{y}$ by Lemma \ref{l:prox_B}, thanks to \cite[Theorem 2.14]{HesLuk13} there exists a neighborhood of $\hat{y}$ on which $P_B$ is almost averaged with violation $\varepsilon$ and averaging constant $1/2$.
Also, $P_A$ is averaged since $A$ is convex. The statement then follows from \cite[Proposition 2]{Tha18}.
\qed
\end{proof}

\begin{definition}[metric subregularity] \cite{DonRoc14}
A set-valued mapping $\Psi:\mathcal{H}\rightrightarrows\mathcal{H}'$ is called \emph{metrically subregular}  at $\hat{y}\in\mathcal{H}$ for $\hat{z}\in \Psi(\hat{y})$ if there exist numbers $\delta>0$ and $\kappa>0$ such that
\[
\kappa\dist(y,\Psi^{-1}(\hat{z})) \le \dist(\hat{z},\Psi(y)),\quad \forall y\in \Ball_{\delta}(\hat{y}).
\]
\end{definition}
Metric subregularity is one of the cornerstones of variational analysis and optimization theory with many important applications, particularly as constraint qualifications for establishing calculus rules for generalized subdifferentials and coderivatives \cite{VA,Mor06.1} and for analyzing stability and convergence of numerical algorithms \cite{DonRoc14,KlaKum02,LukNguTam18}.

We are now ready to formulate another local linear convergence criterion for DRAP in the setting of phase retrieval.

\begin{theorem}[linear convergence of DRAP]\label{t:VA}
Let $\hat{y}\in\mathbb{C}^N$ be a solution to \eqref{PR_Fourier} and suppose that the set-valued mapping $\Psi:=T_{\mathtt{DRAP}}-\Id$ is \emph{metrically subregular} at $\hat{y}$ for $0$.
Then every sequence generated by $T_{\mathtt{DRAP}}$ converges linearly to a fixed point of $T_{\mathtt{DRAP}}$ provided that the initial point is sufficiently close to $\hat{y}$.
\end{theorem}

Compared to \cite[Theorem 2]{Tha18}, Theorem \ref{t:VA} additionally takes the prox-regularity of the sets $A$ and $B$ into account.
The proof is omitted for brevity.

\begin{remark}[necessity of metric subregularity]
There are two types of regularity conditions often required to obtain a convergence result in the nonconvex optimization literature.
The geometry of the phase retrieval problem yields one type of regularity, that is, the prox-regularity of the sets.
The second type of regularity, termed as \textit{metric subregularity}, is difficult to verify, but as been recently shown in \cite{LukTebNgu20} this condition is not only \textit{sufficient} but also {\em necessary} for local linear convergence.
\end{remark}

\begin{remark}[analysis for inconsistent feasibility]\label{r:convergence for inconsistent}
To analyze convergence properties of DRAP in the more challenging setting of inconsistent feasibility (i.e., phase retrieval with noise and model deviation), more technical details are required.
This task can be done by following the lines of \cite{LukNguTam18}, however, the technical assumption of metric subregularity again remain unverifiable in the setting of phase retrieval.
Hence, we chose to formulate the result in the simpler consistent setting.	
\end{remark}

Our goals in the remainder of this section are 1) to link the abstract metric subregularity condition imposed in Theorem \ref{t:VA} to the physical figures of phase retrieval, and 2) to connect the two convergence criteria formulated in Theorems \ref{t:LA} \& \ref{t:VA} by showing that their key assumptions to some extent can be traced back to a common condition on the phase diversities which are the almost only adjustable figures of phase retrieval.

The subsequent analysis is valid only for the phase retrieval setting with two images, that is, we consider $m=2$ in \eqref{matrix M}.
We first recall the concept of transversality.

\begin{definition}[transversality] \cite[page 99]{ClaLedSteWol98}
A pair of sets $\{A,B\}$ is \emph{transversal} at a point $\hat{y}$ in their intersection if 
\begin{equation*}
N_A(\hat{y}) \cap (-N_B(\hat{y})) = \{0\}.
\end{equation*}
\end{definition}

The origin of this concept can be traced back to at least the 19th century in differential geometry which deals with smooth manifolds \cite{GuiPol74}.
We refer the reader to, for example, \cite{Kru06,KruTha15,KruTha16,KruLukNgu18} for various characterizations of transversality and its application in feasibility problem.
The following result shows that the metric subregularity condition in Theorem \ref{t:VA} can be deduced from the transversality property.

\begin{proposition}[transversality implies metric subregularity]\label{p:assymptotic} Let $\hat{y}\in\mathbb{C}^N$ be a solution to \eqref{PR_Fourier} and suppose that the pair of sets $\{A,B\}$ defined in \eqref{A&B} is transversal at $\hat{y}$.
Then, the set-valued mapping $\Psi:=T_{\mathtt{DRAP}}-\Id$ is metrically subregular at $\hat{y}$ for $0$.
\end{proposition}

\begin{proof}
By \cite[Lemma 3]{Tha18} there exist numbers $\delta>0$ and $\kappa>0$ such that
\begin{equation}\label{assymptotic}
\kappa\dist(y,{A\cap B}) \le \norm{y-y^+},\quad  \forall y\in \mathbb{B}_{\delta}(\hat{y}),\; y^+ \in T_{\mathtt{DRAP}}(y).
\end{equation}
Taking the infimum over all $y^+\in T_{\mathtt{DRAP}}(y)$ in the right-hand side of \eqref{assymptotic} and noting that
\[
A\cap B\subset \Fix T_{\mathtt{DRAP}} = \Psi^{-1}(0),
\]
we obtain that
\begin{align*}
\kappa\dist(y,\Psi^{-1}(0)) &\le \kappa\dist(y,{A\cap B})
\\
&\le \dist(y,T_{\mathtt{DRAP}}(y)) = \dist(0,\Psi(y)),\quad  \forall y\in \mathbb{B}_{\delta}(\hat{y}).
\end{align*}
This yields metric subregularity of $\Psi$ at $\hat{y}$ for $0$ as claimed.
\qed
\end{proof}

\begin{remark}\label{r:why m=2}
The restriction $m=2$ involves in Proposition \ref{p:assymptotic} only in an implicit manner.
A further analysis\footnote{It is not presented here for the sake of brevity.} reveals that $m=2$ is a necessary condition for having the transversality assumption fulfilled.
\end{remark}

\begin{proposition}\label{p:transversal of Omega_d}
Let $\hat{y} = M\hat{x} \in\mathbb{C}^N$ be a solution to \eqref{PR_Fourier}.
Then, the pair of sets $\{A,B\}$ defined in \eqref{A&B} is transversal at $\hat{y}$ if and only if the pair of sets $\{\Omega_1,\Omega_2\}$ defined in \eqref{amp_con_set} is transversal at $\hat{x}$.
\end{proposition}

\begin{proof}
Since the sets $A$ and $B$ are respectively the images of the sets $D$ and $\Omega$ defined in \eqref{Omega and D} via the unitary mapping $U$ given by \eqref{matrix U} followed by the scaling of factor $1/\sqrt{2}$, the pair of sets $\{A,B\}$ is transversal at $\hat{y}$ if and only if the pair of sets $\{D,\Omega\}$ is transversal at $(\hat{x},\hat{x})$. 
The latter is in turn equivalent to the transversality of $\{\Omega_1,\Omega_2\}$ at $\hat{x}$ in view of \cite[page 505]{LewLukMal09}.
\qed
\end{proof}

In view of Propositions \ref{p:assymptotic} \& \ref{p:transversal of Omega_d}, the metric subregularity condition in Theorem \ref{t:VA} is guaranteed by the transversality of $\{\Omega_1,\Omega_2\}$ at $\hat{x}$.
In view of \eqref{amp_con_set}, the latter sets are tied to the phase diversities $\{\phi_1,\phi_2\}$ which are represented as unitary matrices $\{D_1,D_2\}$ in \eqref{amp_con_set}.
Unfortunately, the question of choosing $\{\phi_1, \phi_2\}$ such that $\{\Omega_1,\Omega_2\}$ satisfies the transversality at $\hat{x}$ or some weaker property but sufficient for the metric subregularity condition in Theorem \ref{t:VA} is not trivial and open.

We conclude this section with an overall remark on the obtained convergence results.

\begin{remark}\label{r:criteria compared}
Convergence criteria for the DRAP algorithm in Theorems \ref{t:LA} \& \ref{t:VA} are derived from two different analysis approaches\footnote{We are not aware of any other analysis scheme relevant to the phase retrieval problem.}, however, their key assumptions to some extent can be related to the choice of phase diversities.
\end{remark}

\section{Numerical simulation}\label{s:numerical experiment}

We simulate an imaging system with physical parameters as summarized in Table \ref{tbl:sml_parameters}.
The simulation phase $\Phi$ is shown in Figure \ref{fig:sml phase restored} (right).
For the forward model, the amplitude $\chi$ is constant over the pixels in the aperture\footnote{The amplitude is presumed unknown when solving the inverse problem.}.
Five PSF images corresponding to the five phase diversities 
\[
\phi_d = \pi\cdot z_d\cdot Z_2^0,\quad (z_d = -2,-1,0,1,2)
\]
are respectively calculated by
\begin{equation}\label{inten_generate}
p_d = \parennnn{\fft\paren{\chi\cdot \ee^{\jj\paren{\Phi+\phi_d}}}}^2,\quad (d = 1,\ldots,5).
\end{equation}

\begin{table}[h]
	\centering{
		\begin{tabular}[1\baselineskip]{|c|c|c|c|c|c|}
			\hline
			Aperture & Numerical aperture & Wavelength & Pixel size & Image size \\ \hline
			Circular  & 0.25 & 0.633 $\mu m$ & 0.44 $\mu m$ & $256\times 256$\\ \hline
	\end{tabular}}
	\caption{Physical parameters of the simulated imaging system.}
	\label{tbl:sml_parameters}
\end{table}	

We consider a practically relevant case where the images are corrupted with both Poisson and Gaussian noise.
After normalizing the five images generated by \eqref{inten_generate} such that their highest intensities are unity, the normalized images are corrupted with Poisson noise using the Matlab function \textit{imnoise}.
Then, after scaling these noisy images to have the highest intensities of the original ones, we introduce additive white Gaussian noise with signal-to-noise ratio (SNR) 10 dB (decibel) using the Matlab function \textit{awgn}.
The input data images, which are denoted by $r_d$ ($d=1,2,\ldots,5$), are finally obtained by replacing all the negative entries of the corrupted images by zeros.

\begin{figure}[htbp]
	\centering
	\includegraphics[scale=.76]{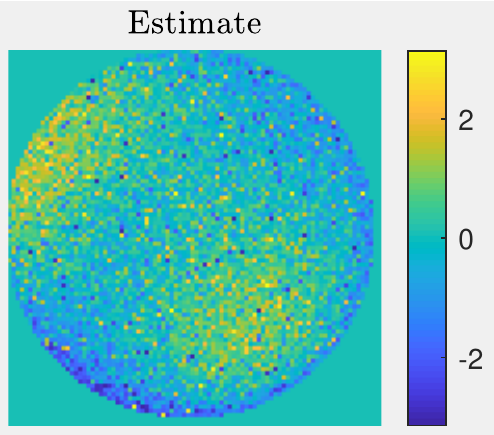}\;
	\includegraphics[scale=.76]{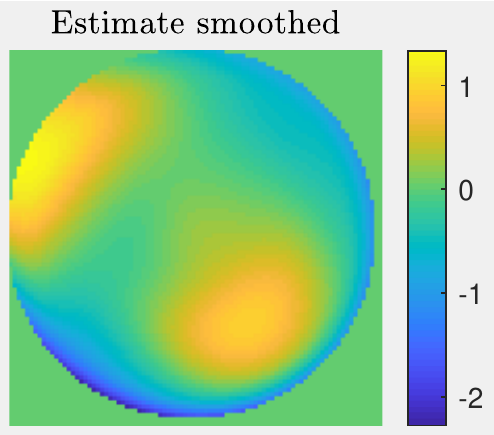}\;
	\includegraphics[scale=.76]{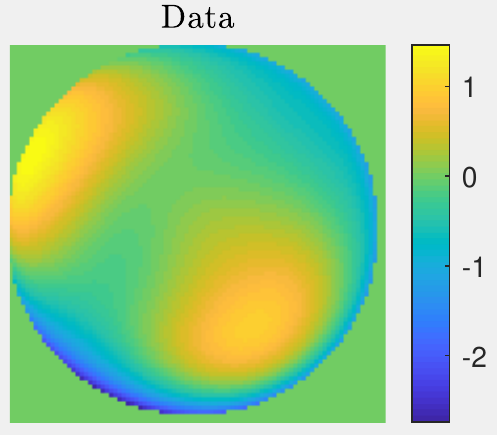}
	\caption{Phase retrieved by DRAP algorithm (left) and its smoothed version using Zernike polynomials (middle) compared to the data phase (right).}
	\label{fig:sml phase restored}
\end{figure}

\begin{figure}[H]
	\centering
	\includegraphics[scale=.6]{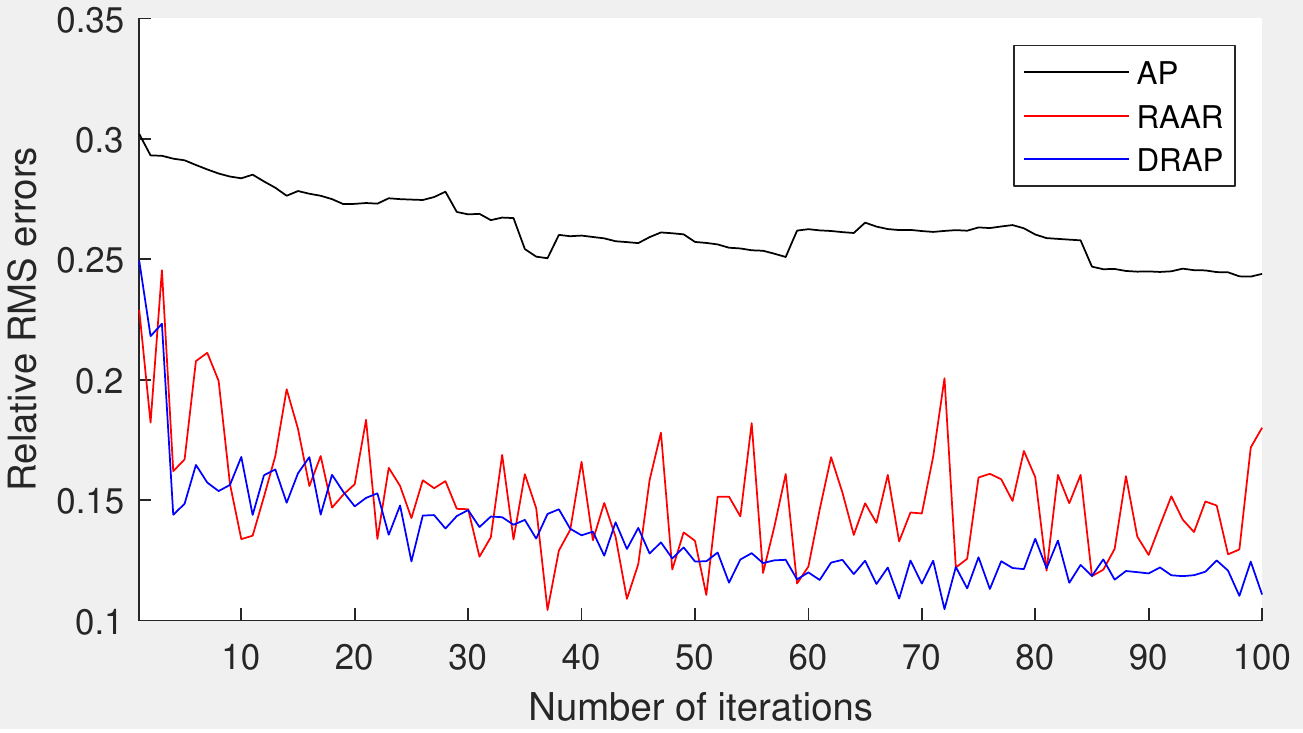}
	\caption{Performance of DRAP (the blue curve) compared to that of AP (the black curve) and RAAR (the black curve) algorithms.}
	\label{fig:simulation data}
\end{figure}

We formulate the phase retrieval problem in the form \eqref{PR_Fourier} and restore the phase $\Phi$ given the five noisy PSF images $r_d$ and the physical parameters specified above using DRAP.
The quality of phase retrieval is measured by the relative \textit{root mean square} (RMS) error of the estimated phase relative to the true phase.
The phase retrieved using 100 iterations of DRAP and 20 iterations of AP\footnote{Additional averaging process using alternating projection is essential since DRAP as well as KMDR, HIO and RAAR does not find an approximate solution to \eqref{PR_Fourier} in a direct manner like the AP method.} is shown in Figure \ref{fig:sml phase restored} (left).
The restored phase is then smoothed using the first 37 Zernike polynomials (in Fringe order convention) and the obtained phase $\widehat\Phi$ is shown in middle figure compared to the data $\Phi$ on the right.
The relative RMS error is $\|\widehat{\Phi}-\Phi\|/\norm{\Phi} = 0.1108$.
Since phase retrieval is ambiguous up to (at least) a \textit{piston term} (global phase shift), the piston terms of the phases are removed before calculating the norms.
 
We compare the performance of DRAP with the other projection algorithms for solving \eqref{PR_Fourier} including AP, KMDR, HIO and RAAR.
The overall results are summarized in Figure \ref{fig:simulation data}.
For brevity, we do not show the results for KMDR and HIO since their performance is far worse than that of RAAR and DRAP in both accuracy and stability.
As shown in Figure \ref{fig:simulation data}, phase retrieval by RAAR and DRAP is at almost the same level of accuracy as well as convergence speed, however, DRAP (the blue curve) is more stable than RAAR (the red curve).
The relaxation parameters used for RAAR and DRAP are 0.8 and 0.45, respectively\footnote{Separate computations are done to determine the seemingly optimal relaxation parameters for all KMDR, HIO, RAAR and DRAP.}.
For completeness, AP (the black curve) is more stable than the other algorithms as expected, however, it is incomparable to RAAR and DRAP in both accuracy and convergence speed.



\end{document}